\documentclass[11pt,reqno]{amsart}

\usepackage{xspace}
\usepackage{amsmath,amsthm,amsfonts,amssymb,latexsym,mathrsfs,color,extarrows}
\usepackage{hyperref}

\def\l{\ell}

\newtheorem{theorem}{Theorem}%[section]
\newtheorem{cor}[theorem]{Corollary}
\newtheorem{prop}[theorem]{Proposition}

\newtheorem{lemma}[theorem]{Lemma}
\newtheorem{definition}[theorem]{Definition}
\newtheorem{ex}[theorem]{Example}

\newcommand{\cpk}{{\rm cpk\,}}

\newcommand{\lpk}{{\rm lpk\,}}
\newcommand{\uprun}{{\rm uprun\,}}
\newcommand{\pk}{{\rm pk\,}}
\newcommand{\des}{{\rm des\,}}

\newcommand{\exc}{{\rm exc\,}}
\newcommand{\cyc}{{\rm cyc\,}}
\newcommand{\fix}{{\rm fix\,}}

\newcommand{\msn}{\mathfrak{S}_n}
\newcommand{\rss}{\mathcal{SS}}
\newcommand{\ms}{\mathfrak{S}}

\newcommand{\rs}{\mathcal{RS}}

\newcommand{\lrf}[1]{\lfloor #1\rfloor}
\newcommand{\lrc}[1]{\lceil #1\rceil}

\newcommand{\sgn}{{\rm sgn\,}}

\DeclareMathOperator{\R}{\mathbb{R}}

\newcommand{\rz}{{\rm RZ}}

\title{The peak statistics on simsun permutations}
\author[S.-M.~Ma]{Shi-Mei Ma}
\address{School of Mathematics and Statistics,
        Northeastern University at Qinhuangdao,
         Hebei 066004, P.R. China}
\email{shimeimapapers@163.com (S.-M. Ma)}
\author[Y.-N. Yeh]{Yeong-Nan Yeh}
\address{Institute of Mathematics,
        Academia Sinica, Taipei, Taiwan}
\email{mayeh@math.sinica.edu.tw (Y.-N. Yeh)}
\subjclass[2010]{Primary 05A05; Secondary 05A15}
\begin{document}

\maketitle
\begin{abstract}
In this paper, we study the relationship among left peaks,
interior peaks and up-down runs of simsun permutations. Properties of the generating polynomials, including the
recurrence relation, generating function and real-rootedness are studied. Moreover, we introduce and study
simsun permutations of the second kind.
%we
%give a constructive proof of a connection between the number of simsun permutations of length $n$ with $k$ descents and
%the number of permutations of $\{1,2,\ldots,n+1\}$ with $k$ interior peaks.
%Furthermore,
\bigskip

\noindent{\sl Keywords}: Simsun permutations; Left peaks; Interior peaks; Alternating runs; Excedances
\end{abstract}
\date{\today}
%\date{\today}
%%%%%%%%%%%%%%%%%%%%%%%%%%%%%%%%%%%%%%%%%%%
\section{Introduction}
%%%%%%%%%%%%%%%%%%%%%%%%%%%%%%%%%%%%%%%%%%%
Let $\msn$ denote the symmetric group of all permutations of $[n]$, where $[n]=\{1,2,\ldots,n\}$.
Let $\pi=\pi(1)\pi(2)\cdots\pi(n)\in\msn$. A descent of $\pi$ is an index $i\in [n-1]$ such that $\pi(i)>\pi(i+1)$.
We say that $\pi$ has no {\it double descents} if there is no
index $i\in [n-2]$ such that $\pi(i)>\pi(i+1)>\pi(i+2)$.
The permutation $\pi$ is called {\it simsun} if for all $k$, the
subword of $\pi$ restricted to $[k]$ (in the order
they appear in $\pi$) contains no double descents. For example,
$35142$ is simsun, but $35241$ is not.
Simsun permutations are useful in describing the action of the symmetric group on the
maximal chains of the partition lattice (see~\cite{Sundaram1994,Sundaram1995}).
They are a variant of Andr\'e permutations that was introduced by Foata and Sch\"utzenberge~\cite{Foata73}.
There has been much recent work related to simsun permutations~(see~\cite{Branden11,Chow11,Deutsch12,Eu14,Foata01,Hetyei98} for instance).

%Recall that a {\it descent} of $\pi$ is a position $i$ such that $\pi(i)>\pi(i+1)$.
%A permutation $\pi\in\msn$ is called an {\it Andr\'e permutation} whenever $\pi$
%has no double descents and satisfies the condition~(see~\cite{Purtill93}):
%For all $1<j<j'\leq n$, if $$\pi(j-1)=\max \{\pi(j-1),\pi(j),\pi(j'-1),\pi(j')\}$$ and $$\pi(j')=\min \{\pi(j-1),\pi(j),\pi(j'-1),\pi(j')\},$$
%then there exists a $j''$, with $j<j''<j'$, such that $\pi(j'')<\pi(j')$.
%The Andr\'e permutations are a variant of simsun permutations and are closely related to the enumeration of the monomials of the cd-index of %$\msn$~(see~\cite{Hetyei96}).
%A {\it simsun} permutation $\pi$ of length $n$ is a permutation of $\msn$ such that
%for all $0\leq k\leq n$ if we remove the $k$ largest entries $n,n-1,\ldots,n-k+1$ from $\pi$,
%the resulting permutation does not have any double descents.
%Simsun permutations
%are named after Rodica Simion and Sheila Sundaram~\cite{Sundaram1995}.
%The simsun permutations are a variant of Andr\'e permutations~(see~\cite{Hetyei96}).

Let $a_i(n)$ be the number of distinct $S_n$-orbits such that the stabiliser of a maximal chain in the orbit is conjugate to the Young subgroup
$S_2^i\times S_1^{n-2i}$. Following Sundaram~\cite[Theorem 3.2]{Sundaram1994}, the numbers $a_i(n)$ satisfy the recurrence relation
$$a_i(n+1)=ia_i(n)+(n-2i+2)a_{i-1}(n),$$
with initial conditions $a_0(1)=1=a_1(2)$, $a_0(n)=0$ for $n>1$ and $a_i(n)=0$ if $2i>n$.
Let $\rs_n$ be the set of simsun permutations of length $n$.
Simion and Sundaram~\cite[p.~267]{Sundaram1994} discovered that $a_i(n)$ is the number of permutations in $\rs_{n-2}$ with $i-1$ descents and $\#\rs_n=E_{n+1}$,
where $E_{n}$ is the $n$th Euler number, which also is the number alternating permutations in $\msn$.

The descent number of $\pi\in\msn$ is defined by $\des(\pi)=\#\{i\in [n-1]: \pi(i)>\pi(i+1)\}$.
Let $S(n,k)=\#\{\pi\in\rs_n: \des(\pi)=k\}$.
We define
$S_n(x)=\sum_{k=0}^{\lrf{n/2}}S(n,k)x^k$.
Then the numbers $S(n,k)$ satisfy the recurrence relation
\begin{equation}\label{Snk-recurrence}
S(n,k)=(k+1)S(n-1,k)+(n-2k+1)S(n-1,k-1),
\end{equation}
with the initial conditions $S(0,0)=1$ and $S(0,k)=0$ for $k\ge 1$, which is equivalent to
\begin{equation}\label{Snx-recu}
S_{n+1}(x)=(1+nx)S_n(x)+x(1-2x)S_n'(x),
\end{equation}
with $S_0(x)=1$. Let $S(x,z)=\sum_{n\geq 0}S_n(x)\frac{z^n}{n!}$.
Chow and Shiu~\cite[Theorem~1]{Chow11} obtained that
\begin{equation}\label{Sxz-expon}
S(x,z)=\left(\frac{\sqrt{2x-1}\sec\left(\frac{z}{2}\sqrt{2x-1}\right)}
{\sqrt{2x-1}-\tan\left(\frac{z}{2}\sqrt{2x-1}\right)}\right)^2.
\end{equation}
For convenience, here we list the first few terms of $S_n(x)$:
\begin{align*}
S_1(x)&=1,
S_2(x)=1+x,
S_3(x)=1+4x,
S_4(x)=1+11x+4x^2,
S_5(x)=1+26x+34x^2.
\end{align*}

%According to~\cite[Theorem~1]{Chow11}, the exponential generating function of $S_n(x)$ is given by
%$$S(x,z)=\sum_{n\geq 0}S_n(x)\frac{z^n}{n!}=\left(\frac{\sqrt{2x-1}\sec(\frac{z}{2}\sqrt{2x-1})}{\sqrt{2x-1}-\tan(\frac{z}{2}\sqrt{2x-1})}\right)^2$$
%
%Following~\cite[Theorem~1]{Chow11}, the exponential generating function of $S_n(x)$ is given by
% $$S(x,t)=\sum_{n\geq %0}S_n(x)\frac{z^n}{n!}=\left(\frac{\sqrt{2x-1}\sec\left(\frac{z}{2}\sqrt{2x-1}\right)}{\sqrt{2x-1}-\tan\left(\frac{z}{2}\sqrt{2x-1}\right)}\right)^2.$$
The number of peaks of permutations is certainly among the most important combinatorial statistics.
See, e.g.,~\cite{Billey13,Ma121,Ma123} and the references therein.
A {\it left peak} in $\pi$ is an index $i\in[n-1]$ such that $\pi(i-1)<\pi(i)>\pi(i+1)$, where we take $\pi(0)=0$.
Let $\lpk(\pi)$ denote the number of
left peaks in $\pi$. For example, $\lpk(21435)=2$.
Sundaram discovered that $a_i(n)$ is also the number of Andr\'e permutations in $\ms_{n-1}$ with $i$ left peaks (see~\cite[p.~175]{Sundaram1996}).
In fact, since any descent of a simsun permutation is a left peak, we have
\begin{equation*}\label{Snxlpk}
S_n(x)=\sum_{\pi\in\rs_n}x^{\lpk(\pi)}.
\end{equation*}

Let $\widehat{W}(n,k)=\#\{\pi\in\msn: \lpk(\pi)=k\}$.
Let
$\widehat{W}_n(x)=\sum_{k\geq0}\widehat{W}({n,k})x^k$.
The polynomials $\widetilde{W}_n(x)$ satisfy the recurrence relation
\begin{equation*}\label{Wnkx-recurrence}
\widehat{W}_{n+1}(x)=(1+nx)\widehat{W}_n(x)+2x(1-x)\widehat{W}_n'(x),
\end{equation*}
with initial conditions $\widehat{W}_0(x)=\widehat{W}_1(x)=1$ (see~\cite[A008971]{Sloane}).
It is well known~\cite[A008971]{Sloane} that %the exponential generating function of $\widehat{W}_n(x)$ is given as follows:
\begin{equation}\label{Wxz-expon}
\widehat{W}(x,z)=\sum_{n\geq 0}\widehat{W}_n(x)\frac{z^n}{n!}
=\frac{\sqrt{1-x}}{\sqrt{1-x}\cosh(z\sqrt{1-x})-\sinh(z\sqrt{1-x})}.
\end{equation}
By comparing~\eqref{Sxz-expon} with~\eqref{Wxz-expon}, we observe that
$S(x,z)=\widehat{W}\left(2x,z/2\right)^2$,
which leads to the following formula:
\begin{equation}\label{SnxWnx}
S_n(x)=\frac{1}{2^n}\sum_{k=0}^n\binom{n}{k}\widehat{W}_k(2x)\widehat{W}_{n-k}(2x).
\end{equation}

Denote by $B_n$ the hyperoctahedral group of rank $n$. Elements $\pi$ of $B_n$ are signed permutations of the set $\pm[n]$ such that $\pi(-i)=-\pi(i)$ for all $i$, where $\pm[n]=\{\pm1,\pm2,\ldots,\pm n\}$.
A {\it snake} of type $B_n$ is a signed permutation $\pi(1)\pi(2)\cdots\pi(n)\in\ B_n$ such that $0<\pi(1)>\pi(2)<\cdots\pi(n)$.
The $n$th {\it Springer number} $S_n$ is the number of snakes of of type $B_n$.
Springer~\cite{Springer} derived the following generating function:
$$\sum_{n\geq 0}S_n\frac{z^n}{n!}=\frac{1}{\cos z-\sin z},$$
which equals $\widehat{W}(2,z)$. As a special case of~\eqref{SnxWnx}, we get
$$E_{n+1}=\frac{1}{2^n}\sum_{k=0}^n\binom{n}{k}S_kS_{n-k}.$$
We refer the reader to~\cite{Chen11} for various structures related to Springer numbers.
Motivated by~\eqref{SnxWnx}, it is natural to study peak statistics on simsun permutations.

This paper is organized as follows.
In Section~\ref{Section-2}, we give a constructive proof of a connection
between $S(n,k)$ and the number of permutations in $\ms_{n+1}$ with $k$ interior peaks.
%The numbers $S(n,k)$ arise often in combinatorics and other branches of mathematics~\cite[A094503,~A113897]{Sloane}
In Section~\ref{Section-3}, we count simsun permutations by their interior peaks. In Section~\ref{Section-4}, we count simsun permutations by their up-down runs. In Section~\ref{Section-5}, we introduce
simsun permutations of the second kind.
%%%%%%%%%%%%%%%%%%%%%%%%%%%%%%%%%%%%%%%%%%%
\section{Relationship to permutations with a given number of interior peaks}\label{Section-2}
%\hspace*{\parindent}
%%%%%%%%%%%%%%%%%%%%%%%%%%%%%%%%%%%%%%%%%%
We first recall some basic definitions of peak statistics.
An {\it interior peak} in $\pi$ is an index $i\in\{2,3,\ldots,n-1\}$
such that $\pi(i-1)<\pi(i)>\pi(i+1)$ (see~\cite[A008303]{Sloane}). Let $\pk(\pi)$ denote the number of interior peaks
in $\pi$. Let $W(n,k)=\#\{\pi\in\msn: \pk(\pi)=k\}$.
We define $W_n(x)=\sum_{k\geq 0}W(n,k)x^k$.
The polynomials $W_n(x)$ satisfy the recurrence relation
\begin{equation*}
W_{n+1}(x)=(nx-x+2)W_n(x)+2x(1-x)W'_n(x),
\end{equation*}
with initial conditions $W_1(x)=1,W_2(x)=2$ and $W_3(x)=4+2x$.
%It is well known~\cite[A008303]{Sloane} that
%\begin{equation*}
% W(x,z)=\sum_{n\geq 1}W_n(x)\frac{z^n}{n!}=\frac{\sinh(z\sqrt{1-x})}{\sqrt{1-x}\cosh(z\sqrt{1-x})-\sinh(z\sqrt{1-x})}.
%\end{equation*}
We say that $\pi$ changes direction at position $i$ if either $\pi({i-1})<\pi(i)>\pi(i+1)$, or
$\pi(i-1)>\pi(i)<\pi(i+1)$, where $i\in\{2,3,\ldots,n-1\}$. We say that $\pi$ has $k$ {\it alternating runs} if there are $k-1$ indices $i$ where $\pi$ changes direction~(see~\cite[A059427]{Sloane}).
Let $R(n,k)$ denote the number of permutations in $\msn$ with $k$ alternating runs and let $R_n(x)=\sum_{k=1}^{n-1}R(n,k)x^k$.
The alternating runs of permutations was first studied by Andr\'e~\cite{Andre84} and he
showed that $R(n,k)$ satisfies the following recurrence relation
\begin{equation*}\label{rnk-recurrence}
R(n,k)=kR(n-1,k)+2R(n-1,k-1)+(n-k)R(n-1,k-2)
\end{equation*}
for $n,k\geqslant 1$, where $R(1,0)=1$ and $R(1,k)=0$ for $k\geqslant 1$.

Let $d(n,k)$ denote the number of increasing 1-2 trees on $[n]$ with $k$ leaves (see~\cite[A094503]{Sloane}).
Let $D_n(x)=\sum_{k\geq 1}d(n,k)x^k$. It follows from~\cite[Proposition 4]{Chow11} that
$D_{n+1}(x)=xS_{n}(x)$ for $n\ge 0$.
%Hence $S(n,k)$ is also the number of increasing 0-1-2 trees of order $n+1$ with $k+1$ leaves.
Using~\cite[Corollary 2,~Theorem 11]{Ma123}, we get
\begin{equation}\label{Rnx-Wnx}
R_n(x)=\frac{x(1+x)^{n-2}}{2^{n-2}}W_n\left(\frac{2x}{1+x}\right)=2x(1+x)^{n-2}S_{n-1}\left(\frac{x}{1+x}\right)\quad\textrm{for $n\ge 2$}.
\end{equation}
%Andr\'e~\cite{Andre84}
Therefore, combining~\eqref{Rnx-Wnx} and~\cite[Eq.~(13)]{Ma122}, we get the following result.
\begin{prop}
For $n\geq 1$ and $0\leq k\leq \lrf{\frac{n}{2}}$, we have
%we have $W_{n+1}(2x)=2^nS_n(x)$.
\begin{equation}\label{WnkSnk}
W(n+1,k)=2^{n-k}S(n,k).
\end{equation}
Moreover, we have
$$S_n(x)=\frac{1}{2^{n+1}x}\sum_{k=0}^{\lrf{{n}/{2}}+1}p(n+1,n-2k+2)(2x-1)^k$$
for $n\geq 1$,
where
\begin{equation*}
p(n,n-2k+1)=(-1)^{k}\sum_{i\geq 1}i!{n \brace i}(-2)^{n-i}\left[\binom{i}{n-2k}-\binom{i}{n-2k+1}\right].
\end{equation*}
\end{prop}

In the rest of this section, we give a constructive proof of~\eqref{WnkSnk}.
Let $$D(\pi)=\{i\in[n-2]: \pi(i)>\pi(i+1)\}$$ be the descent set of $\pi\in\rs_{n}$. Set $\overleftarrow{D}(\pi)=\{i-1: i\in D(\pi)\}$.
It should be noted that if we get a permutation $\pi'\in\rs_{n+1}$ from a permutation $\pi\in\rs_{n}$ by inserting the entry $n+1$ into $\pi$, then the entry $n+1$ can not be inserted right after $\pi(j)$, where $j\in \overleftarrow{D}(\pi)$.
In this section, we always assume that permutations in $\rs_{n}$ are prepended by 0.
That is, we identify a permutation $\pi(1)\pi(2)\cdots\pi(n)\in\rs_n$ with the word $\pi(0)\pi(1)\pi(2)\cdots\pi(n)$, where $\pi(0)=0$.
Define $\rs_{n,k}=\{\pi\in\rs_n \mid \des(\pi)=k\}$.
We can now introduce a definition of {\it labeled simsun permutations}.
\begin{definition}\label{def01}
Let $\pi\in\rs_{n,k}$.
Suppose $i_1<i_2<\cdots<i_k$ are elements of $D(\pi)$. Then we put the superscript label $x_{r}$
right after $\pi(i_{r})$, where $1\leq r\leq k$.
%Moreover, we put the superscript label $x_{k+1}$ in the end of $\pi$.
 If $j_1<j_2<\cdots<j_{n-2k}$ are elements of the set $\{0,1,2,\ldots,n-1\}\setminus(D(\pi)\cup \overleftarrow{D}(\pi))$,
then we put the superscript label $y_{s}$
right after $\pi(j_s)$, where $1\leq s\leq n-2k$.
\end{definition}
%If we insert the entry $n+1$ into $\pi\in\rs_{n,k}$ with
%superscript label $y_s$ (resp. $x_r$), then the insertion of $n+1$ does (resp. does not) form a new descent.

Let $\ms_{n,k}=\{\pi\in\msn \mid \pk(\pi)=k\}$. We introduce a definition of {\it labeled permutations}.
\begin{definition}
Let $\pi\in\ms_{n,k}$.
Suppose $i_1<i_2<\cdots<i_k$ are indices of the interior peaks of $\pi$. Then we put the superscript labels $p_r$
immediately before and right after $\pi(i_{r})$, where $1\leq r\leq k$.
%Moreover, we put the superscript label $x_{k+1}$ in the end of $\pi$.
 If $j_1<j_2<\cdots<j_{n-2k-1}$ are elements of the set
 $\{1,2,3,\ldots,n-1\}\setminus (\{i_1,i_2,\ldots,i_k\}\cup \{i_1-1,i_2-1,\ldots,i_k-1\})$,
then we put the superscript label $q_{s}$
right after $\pi(j_s)$, where $1\leq s\leq n-2k-1$.
\end{definition}

In the following discussion, we always add labels to permutations in $\rs_{n,k}$ and $\ms_{n,k}$.
As an example, for $\pi=34125$, if we say that $\pi\in \rs_{5,1}$, then the labels of $\pi$
is given by $^{y_1}34^{x_1}1^{y_2}2^{y_3}5$; if we say that $\pi\in\ms_{5,1}$, then the labels of $\pi$
is given by $3^{p_1}4^{p_1}1^{q_1}2^{q_2}5$.

Now we construct a correspondence, denoted by $\Phi$, between $\rs_{n,k}$ and $\ms_{n+1,k}$.
When $n=1$, the correspondence between $\rs_{1,0}$ and $\ms_{2,0}$ is given by
$$^{y_1}1 \xlongleftrightarrow{\Phi} \{1^{q_1}2,2^{q_1}1\}.$$
When $n=2$, the correspondence between $\rs_{2,k}$ and $\ms_{3,k}$ is given by
\begin{align*}
^{y_1}1^{y_2}2&\xlongleftrightarrow{\Phi} \{1^{q_1}2^{q_2}3,3^{q_1}1^{q_2}2,2^{q_1}1^{q_2}3,3^{q_1}2^{q_2}1\};\\
2^{x_1}1& \xlongleftrightarrow{\Phi} \{1^{p_1}3^{p_1}2,2^{p_1}3^{p_1}1\}.
\end{align*}
Let $n=m$. Suppose $\Phi$ is a correspondence between $\rs_{m,k}$ and $\ms_{m+1,k}$ for all $k$.
More precisely, given an element $\pi\in\rs_{m,k}$. Suppose we have the correspondence
\begin{align*}
\pi&\xlongleftrightarrow{\Phi}\{\sigma_1,\sigma_2,\ldots,\sigma_{2^{m-k}}\},
\end{align*}
where $\sigma_i\in\ms_{m+1,k}$ for $1\leq i\leq 2^{m-k}$.
Consider the case $n=m+1$.
Suppose $\widehat{\pi}\in\rs_{m+1}$ is obtained from $\pi$ by inserting the entry $m+1$ into $\pi$.
We distinguish three cases:
\begin{enumerate}
  \item [\rm ($i$)] If $\widehat{\pi}(m+1)=m+1$, then we insert the entry $m+2$ at the front or at the end of each $\sigma_i$.
  In this case, the obtained elements in $\Phi(\widehat{\pi})$ all have $k$ interior peaks. Therefore, we get $2\cdot 2^{m-k}=2^{m+1-k}$ elements in $\ms_{m+2,k}$.
 \item [\rm ($ii$)] If the entry $m+1$ is inserted to the position of $\pi$ with label $x_r$,
 then we insert the entry $m+2$ to one of the positions of each $\sigma_i$ with label $p_r$. In this case, $\des(\widehat{\pi})=k$ and we get $2\cdot 2^{m-k}=2^{m+1-k}$ elements in $\ms_{m+2,k}$.
 \item [\rm ($iii$)] If the entry $m+1$ is inserted to the position of $\pi$ with label $y_s$,
 then we insert the entry $m+2$ to the position of each $\sigma_i$ with label $q_s$.
 In this case, $\des(\widehat{\pi})=k+1$ and we get $2^{m-k}=2^{(m+1)-(k+1)}$ elements in $\ms_{m+2,k+1}$.
  \end{enumerate}
%After the above step, we label $\widehat{\pi}$ and these permutations in the set $\Phi(\widehat{\pi})$.
It is straightforward to show that each labeled permutation in $\Phi(\widehat{\pi})$ will be obtained exactly once in this way.
Conversely, given an element $\tau$ of $\ms_{m+2,k}$. Removing the entry $m+2$ of $\tau$, we can find the position of the largest entry of the corresponding
simsun permutation in $\rs_{m+1}$.
As illustrated in example~\ref{example01}, we can get an unique element of
$\rs_{m+1}$ by repeatedly removing the largest entry.
By induction, we see that $\Phi$ is the desired correspondence between $\rs_{m,k}$ and $\ms_{m+1,k}$, which also
gives a constructive proof of~\eqref{WnkSnk}.

%Here we present an example for illustration the correspondence.
\begin{ex}\label{example01}
Given $\pi=3412\in\rs_{4,1}$. The correspondence between $\pi$ and $\Phi(\pi)$ is built up as
follows:
\begin{align*}
^{y_1}1~&\xlongleftrightarrow{\Phi} \{1^{q_1}2,2^{q_1}1\};\\
^{y_1}1^{y_2}2&\xlongleftrightarrow{\Phi} \{1^{q_1}2^{q_2}3,3^{q_1}1^{q_2}2,2^{q_1}1^{q_2}3,3^{q_1}2^{q_2}1\};\\
3^{x_1}1^{y_1}2~&\xlongleftrightarrow{\Phi} \{1^{p_1}4^{p_1}2^{q_1}3,3^{p_1}4^{p_1}1^{q_1}2,2^{p_1}4^{p_1}1^{q_1}3,3^{p_1}4^{p_1}2^{q_1}1\};\\
^{y_1}34^{x_1}1^{y_2}2~&\xlongleftrightarrow{\Phi} \{1^{p_1}5^{p_1}4^{q_1}2^{q_2}3,3^{p_1}5^{p_1}4^{q_1}1^{q_2}2,2^{p_1}5^{p_1}4^{q_1}1^{q_2}3,3^{p_1}5^{p_1}4^{q_1}2^{q_2}1,\\
     ~~~~~~~&1^{q_1}4^{p_1}5^{p_1}2^{q_2}3,3^{q_1}4^{p_1}5^{p_1}1^{q_2}2,2^{q_1}4^{p_1}5^{p_1}1^{q_2}3,3^{q_1}4^{p_1}5^{p_1}2^{q_2}1\}.
%34125~&\xlongleftrightarrow{\Phi} \{154236,354126,254136,354216, 615423,635412,625413,635421,\\
% ~~~~~~~&145236,345126,245136,345216, 614523,634512,624513,634521\}.
%(1~^{-1}~3~^{-2}~5~^{1}~2^{2})(4^{3})&\rightarrow (1,3,5,2,6)(4) \Leftrightarrow y_6=3~\textrm{or}~ y_6=4.
\end{align*}
\end{ex}
%%%%%%%%%%%%%%%%%%%%%%%%%%%%%%%%%%%%%%%%%%%
\section{Interior peaks of simsun permutations}\label{Section-3}
%\hspace*{\parindent}
%%%%%%%%%%%%%%%%%%%%%%%%%%%%%%%%%%%%%%%%%%
Let $\rs_n^{+}=\{\pi\in\rs_n: \pi(1)>\pi(2)\}$ and $\rs_n^{-}=\{\pi\in\rs_n: \pi(1)<\pi(2)\}$.
For $\pi\in\rs_n^{+}$, we have $\lpk(\pi)=\pk(\pi)+1$. While for $\pi\in\rs_n^{-}$, we have
$\lpk(\pi)=\pk(\pi)$.

We define
\begin{align*}
P_n(x)&=\sum_{\pi\in\rs_n}x^{\pk(\pi)}=\sum_{k\geq 0} P(n,k)x^k,\\
P^{+}_n(x)&=\sum_{\pi\in\rs_n^{+}}x^{\pk(\pi)}=\sum_{k\geq 0} P^{+}(n,k)x^k, \\
P^{-}_n(x)&=\sum_{\pi\in\rs_n^{-}}x^{\pk(\pi)}=\sum_{k\geq 0} P^{-}(n,k)x^k.
\end{align*}

The following lemma is a fundamental result.
\begin{lemma}\label{lemma01}
For $n\geq 2$, we have
\begin{equation}\label{Pnk-recu01}
P^+(n+1,k)=(k+1)P^+(n,k)+(n-2k)P^+(n,k-1)+P^-(n,k),
\end{equation}
\begin{equation}\label{Pnk-recu02}
P^-(n+1,k)=(k+1)P^-(n,k)+(n-2k+1)P^-(n,k-1)+P^+(n,k-1).
\end{equation}
Equivalently,
the polynomials $P_{n}^+(x)$ and $P_{n}^-(x)$ satisfy the following recurrence relations
\begin{align*}
P_{n+1}^+(x)&=((n-2)x+1)P_{n}^+(x)+x(1-2x)\frac{d}{dx}P_{n}^+(x)+P_{n}^-(x),\\
P_{n+1}^-(x)&=((n-1)x+1)P_{n}^-(x)+x(1-2x)\frac{d}{dx}P_{n}^-(x)+xP_{n}^+(x).
\end{align*}
%for $n\geq 2$, with initial conditions $P^{+}_1(x)=P^{+}_2(x)=P^{-}_1(x)=P^{-}_2(x)=1$.
\end{lemma}
\begin{proof}
We now prove~\eqref{Pnk-recu01}.
There are three ways we can get a permutation $\pi'\in\rs_{n+1}^+$ with $k$ interior peaks from a permutation $\pi\in\rs_{n}$
by inserting the entry $n+1$ into $\pi$:
\begin{enumerate}
\item [(a)] If $\pi\in \rs_n^+$ and $\pk(\pi)=k$, then we can insert the entry $n+1$ right after an interior peak of $\pi$ or put the entry $n+1$ at the end of $\pi$.
As we have $P^+(n,k)$ choices for $\pi$, this accounts for $(k+1)P^+(n,k)$ possibilities.
\item [(b)] If $\pi\in \rs_n^+$ and $\pk(\pi)=k-1$, then there are $n-2k$ positions could be inserted the entry $n+1$, since we cann't insert the entry $n+1$ immediately before or right after each left peak of $\pi$.
 As we have $P^+(n,k-1)$ choices for $\pi$, this accounts for $(n-2k)P^+(n,k-1)$ possibilities.
 \item [(c)] If $\pi\in \rs_n^-$ and $\pk(\pi)=k$, then we have to put the entry $n+1$ at the front of $\pi$.
 \end{enumerate}
This completes the proof of~\eqref{Pnk-recu01}. In the same way, one can get~\eqref{Pnk-recu02}.
\end{proof}

The first few terms of the $P_n(x),P^{+}_n(x)$ and $P^{-}_n(x)$ are respectively given as follows:
\begin{align*}
P_1(x)& =1,P_2(x) =2,P_3(x)=3+2x,P_4(x)=4+12x,P_5(x)=5+44x+12x^2;\\
P^{+}_1(x)& =1,P^{+}_2(x)=1,P^{+}_3(x)=2,P^{+}_4(x)=3+4x,P^{+}_5(x)=4+22x;\\
P^{-}_1(x)& =1,P^{-}_2(x) =1,P^{-}_3(x)=1+2x,P^{-}_4(x)=1+8x,P^{-}_5(x)=1+22x+12x^2.
\end{align*}
By Lemma~\ref{lemma01}, it is easy to deduce that $$\deg P^{+}_n(x)=\lrf{{(n-2)}/{2}},~\deg P_n(x)=\deg P^{-}_n(x)=\lrf{{(n-1)}/{2}}.$$

\begin{lemma}\label{lemma02}
For $n\geq 1$, we have
\begin{equation}\label{PnkSnk01}
P^{+}(n+1,k)=(n-2k)S(n,k),~P^{-}(n+1,k)=(1+k)S(n,k).
\end{equation}
Equivalently,
\begin{equation}\label{PnkxSnkx01}
P_{n+1}^+(x)=nS_n(x)-2xS_n'(x),~P_{n+1}^-(x)=S_n(x)+xS_n'(x).
\end{equation}
\end{lemma}
\begin{proof}
We prove~\eqref{PnkSnk01} by induction. If $n=1$, the result
is obvious, so we proceed to the inductive step.
Suppose the result holds for $n=m$. For $n=m+1$, combining~\eqref{Snk-recurrence} and~\eqref{Pnk-recu01}, we have
\begin{align*}
P^{+}(m+2,k)&=(k+1)P^+(m+1,k)+(m+1-2k)P^+(m+1,k-1)+P^-(m+1,k)\\
&=(k+1)(m-2k)S(m,k)+(m+1-2k)(m+2-2k)S(m,k-1)+(k+1)S(m,k)\\
&=(k+1)(m-2k)S(m,k)+(m+1-2k)(m+2-2k)S(m,k-1)+\\
&S(m+1,k)-(m+2-2k)S(m,k-1)\\
&=(m-2k)[(k+1)S(m,k)+(m+2-2k)S(m,k-1)]+S(m+1,k)\\
&=(m-2k)S(m+1,k)+S(m+1,k)\\
&=(m+1-2k)S(m+1,k).
\end{align*}
Along the same lines, one can prove $P^{-}(n+1,k)=(1+k)S(n,k)$.
\end{proof}

It is clear that $P(n,k)=P^+(n,k)+P^{-}(n,k)$ for $n\geq 2$. We can now conclude the following result from the discussion above.
\begin{theorem}
For $n\geq 1$, we have
\begin{equation}\label{PnkSnk02}
P(n+1,k)=(n+1-k)S(n,k),
\end{equation}
or equivalently,
\begin{equation}\label{PnkxSnkx02}
P_{n+1}(x)=(n+1)S_n(x)-xS_n'(x).
\end{equation}
Furthermore, we have
\begin{equation}\label{pnk-recurrence}
P(n+1,k)=\frac{(k+1)(n-k+1)}{n-k}P(n,k)+(n-2k+1)P(n,k-1)
\end{equation}
for $0\leq k\leq \lrf{n/2}$.
In particular, $P(n,0)=n$ and $P(n,1)=(n-1)(2^{n-1}-n)$ for $n\geq 1$.
\end{theorem}

It should be noted that~\eqref{pnk-recurrence} follows immediately from
~\eqref{Snk-recurrence} and~\eqref{PnkSnk02}.
%For $1\leq n\leq 7$, the numbers $P(n,k)$ can be arranged as follows with $P(n,k)$ in row $n$ and column $k$:
%$$\begin{array}{ccccc}
%  1 &  &  &  & \\
%  2 &  &  &  & \\
%  3 & 2 &  &  & \\
%  4 & 12 &  &  & \\
%  5 & 44 & 12 &  &  \\
%  6 & 130 & 136 &  & \\
%  7 & 342&900 &136&
%\end{array}$$

We now recall some notations from~\cite{gw96} concerning the zeros of polynomials.
%A polynomial $p\in\R[x]$ is {\it standard} if its leading coefficient is positive.
Let $\rz$ denote the set of real polynomials with only real zeros.
Furthermore, denote by $\rz(I)$ the set of such polynomials all
whose zeros are in the interval $I$.
Suppose that $p,q\in\rz$,
that those of $p$ are $\xi_1\leqslant\cdots\leqslant\xi_n$,
and that those of $q$ are $\theta_1\leqslant\cdots\leqslant\theta_m$.
We say that $p$ {\it interlaces} $q$ if $\deg q=1+\deg p$ and the zeros of
$p$ and $q$ satisfy
$$
\theta_1\leqslant\xi_1\leqslant\theta_2\leqslant\cdots\leqslant\xi_n
\leqslant\theta_{n+1}.
$$
We also say that $p$ {\it alternates left of} $q$ if $\deg p=\deg q$
and the zeros of $p$ and $q$ satisfy
$$
\xi_1\leqslant\theta_1\leqslant\xi_2\leqslant\cdots\leqslant\xi_n
\leqslant\theta_n.
$$
%More generally, we say that $p$ interlaces (resp., alternates left of)
%$q$ in $I\subseteq\R$ if the zeros of $p$ and $q$ in $I$ satisfy the
%corresponding property.
%This refined notion of interlacing and alternating properties will
%make statements of results handy.
We use the notation $p\dagger q$ for ``$p$ interlaces $q$,''
$p\ll q$ for ``$p$ alternates left of $q$,'' and $p\prec q$ for
either $p\dagger q$ or $p\ll q$. For notational convenience,
let $a\prec bx+c$ for any real constants $a,b,c$.

We now recall a result on the real-rootedness of $S_n(x)$.
%Here we recall a result of~\cite{Chow11}.
\begin{lemma}[{\cite[Theorem 1]{Chow11}}]
For $n\geq 2$, we have $S_n(x)\in\rz(-\infty,0)$ and $S_n(x)\prec S_{n+1}(x)$.
\end{lemma}

Let $\sgn$ denote the sign function defined on $\R$.
\begin{theorem}
For $n\geq 2$, we have $P_n(x),P^{+}_n(x),P^{-}_n(x)\in\rz(-\infty,0)$ and
  $$P_{n+1}(x)\ll S_n(x),~ P^{+}_{n+1}(x)\prec S_n(x),~S_n(x)\ll P^{-}_{n+1}(x).$$
\end{theorem}
\begin{proof}
For $n\geq 2$, let $r_{\lrf{n/2}}<r_{\lrf{n/2}-1}<\cdots<r_2<r_1$ be the distinct zeros of $S_n(x)$.
Then by~\eqref{PnkxSnkx02}, we get $\sgn P_{n+1}(r_i)=(-1)^{i-1}$ for $i=1,2,\ldots,\lrf{n/2}$.
Hence $P_{n+1}(x)$ has precisely one zero in each of $\lrf{n/2}-1$ intervals $(r_{\lrf{n/2}},r_{\lrf{n/2}-1}),\ldots,(r_2,r_1)$.
Recall that $\deg P_{n+1}(x)=\lrf{{n}/{2}}$.
If $n=2k+1$ is odd, then $\sgn P_{n+1}(r_k)=(-1)^{k-1}$ and
$\sgn P_{n+1}(-\infty)=(-1)^k$. If $n=2k+2$ is even, then $\sgn P_{n+1}(r_k)=(-1)^{k}$ and
$\sgn P_{n+1}(-\infty)=(-1)^{k+1}$. Thus $P_{n+1}(x)$ has an additional zero in the interval $(-\infty,r_{\lrf{n/2}})$.
Therefore, we have $P_{n+1}(x)\ll S_n(x)$.
Similarly, by using~\eqref{PnkxSnkx01}, one can derive $P^{+}_{n+1}(x)\prec S_n(x)$ and $S_n(x)\ll P^{-}_{n+1}(x)$.
\end{proof}
%%%%%%%%%%%%%%%%%%%%%%%%%%%%%%%%%%%%%%%%%%
\section{Up-down runs of simsun permutations}\label{Section-4}
%\hspace*{\parindent}
%%%%%%%%%%%%%%%%%%%%%%%%%%%%%%%%%%%%%%%%%%
An {\it alternating subsequence} of $\pi\in\msn$ is a subsequence $\pi(i_1),\pi(i_2),\ldots,\pi(i_k)$ satisfying
$$\pi(i_1)>\pi(i_2)<\pi(i_3)>\cdots\pi(i_k),$$
where $i_1<i_2<\cdots <i_k$. Motivated by the study of longest increasing subsequences,
Stanley~\cite{Sta08} studied the longest alternating subsequences.
Let $\l_n(\pi)$ be the length of the longest alternating subsequence of a permutation $\pi\in\msn$.
%From then on, there is a rich literature devoted to the length (number of terms) of the longest alternating subsequence of permutations (see~\cite{Firro07,Pak15} for instance).
The {\it up-down runs} of a permutation $\pi$ are the alternating runs of $\pi$ endowed with a 0
in the front~(see~\cite[A186370]{Sloane}). For example, the permutation $\pi=514623$ has 4 alternating runs and 5 up-down runs.
%Clearly, the number of up-down runs of $\pi$ equals the length of the longest alternating subsequence of $\pi$.
Let $\uprun(\pi)$ be the number of up-down runs of $\pi$. It is clear that $\uprun(\pi)=\l_n(\pi)$ for $\pi\in\msn$.
We define $$T_n(x)=\sum_{\pi\in\rs_n}x^{\uprun(\pi)}=\sum_{k=1}^nT(n,k)x^k.$$
The first few terms of $T_n(x)$ are
\begin{align*}
  T_1(x) =x,~T_2(x) =x+x^2,~T_3(x)=x+2x^2+2x^3,~T_4(x)=x+3x^2+8x^3+4x^4.
 % T_5(x)&= x+4x^2+22x^3+22x^4+12x^5.
\end{align*}
\begin{theorem}
For $n\geq 1$, the numbers $T(n,k)$ satisfy the recurrence relation
\begin{equation}\label{Tnk-recu}
T(n,k)=\lrc{k/2}T(n-1,k)+T(n-1,k-1)+(n-k+1)T(n-1,k-2),
\end{equation}
with initial conditions $T(0,0)=1$ and $T(0,k)=0$ for $k>0$.
\end{theorem}
\begin{proof}
There are three ways in which a permutation $\pi'\in\rs_{n}$ with $\uprun(\pi')=k$
can be obtained from a permutation $\pi\in\rs_{n-1}$ by inserting the entry $n$ into $\pi$.
\begin{enumerate}
\item [(a)] If $\uprun(\pi)=k$, then we can insert the entry $n$ right after the end of
each ascending run. This accounts for $\lrc{k/2}T(n-1,k)$ possibilities.
\item [(b)] If $\uprun(\pi)=k-1$, then we distinguish two cases: when $\pi$ ends with an
ascending run, we insert the entry $n$ to the front of the last entry of $\pi$; when $\pi$ ends with
descending run, we insert the entry $n$ at the end of $\pi$.
This gives $T(n-1,k-1)$ possibilities.
\item [(c)] If $\uprun(\pi)=k-2$,
then we can insert the entry $n$ into the remaining $n-k+1$ positions.
This gives $(n-k+1)T(n-1,k-2)$ possibilities.
\end{enumerate}
This completes the proof of~\eqref{Tnk-recu}.
\end{proof}

Note that $S(n,0)=T(n,1)=1$,
corresponding to the permutation $12\cdots n$. Recall that an element $\pi$ of $\msn$ is {\it alternating} if $\pi(1)>\pi(2)<\pi(3)>\cdots \pi(n)$. In other words, $\pi(i)<\pi({i+1})$ if $i$ is even and $\pi(i)>\pi({i+1})$ if $i$ is odd.
If $n=2m$ is even, then $S(2m,m)=T(2m,2m)$, corresponding to
the number of alternating permutations in $\rs_{2m}$. If $n=2m+1$ is odd, applying the complement operation $\phi$ to $\pi\in\rs_n$, i.e., $\phi(\pi(i))=\pi(n+1-i)$,
it is clear that $P(2m+1,m)=T(2m+1,2m+1)$ counts the number of alternating permutations in $\rs_{2m+1}$.
In general, by analyzing permutations in $\rs_n^{+}$ and $\rs_n^{-}$, it is easy to verify that
\begin{equation*}\label{SnkTnkPnk}
S(n,k)=T(n,2k)+T(n,2k+1),~P(n,k)=T(n,2k+1)+T(n,2k+2).
\end{equation*}
Equivalently, we have
\begin{equation}\label{TnxSnxPnx}
(1+x)T_n(x)=xS_n(x^2)+x^2P_n(x^2)\quad\textrm{for $n\ge 1$.}
\end{equation}

\begin{theorem}
For $n\geq 0$, we have
\begin{equation*}
T_{n+1}(x)=x(1+nx)S_n(x^2)+\frac{1}{2}x^2(1-2x)S_n'(x^2).
\end{equation*}
%and $P_n(x)\prec S_n(x)$.
\end{theorem}
\begin{proof}
It follows from~\eqref{Snx-recu},~\eqref{PnkxSnkx02} and~\eqref{TnxSnxPnx} that
\begin{align*}
 (1+x)T_{n+1}(x)&=xS_{n+1}(x^2)+x^2P_{n+1}(x^2) \\
  & =x((1+nx^2)S_n(x^2)+\frac{1}{2}x(1-2x^2)S_n'(x^2))+x^2((n+1)S_n(x^2)-\frac{x}{2}S_n'(x^2))\\
  & =x(1+(n+1)x+nx^2)S_n(x^2)+\frac{1}{2}x^2(1-x-2x^2)S_n'(x^2)\\
  &=x(1+x)(1+nx)S_n(x^2)+\frac{1}{2}x^2(1+x)(1-2x)S_n'(x^2),
\end{align*}
the statement immediately follows.
\end{proof}

We call the simsun permutations discussed above to be the simsun permutations of the first kind. In the next section, we shall introduce
the simsun permutations of the second kind.
%%%%%%%%%%%%%%%%%%%%%%%%%%%%%%%%%%%%%%%%%%%%%%%%%%%%%%%%%%%%%%%%%%%%%%%%%%
%%%%%%%%%%%%%%%%%%%%%%%%%%%%%%%%%%%%%%%%%%%%%%%%%%%%%%%%%%%%%%%%%%%%%%%%%%
%%%%%%%%%%%%%%%%%%%%%%%%%%%%%%%%%%%%%%%%%%%%%%%%%%%%%%%%%%%%%%%%%%%%%%%%%%
\section{Simsun permutations of the second kind}\label{Section-5}
%%%%%%%%%%%%%%%%%%%%%%%%%%%%%%%%%%%%%%%%%%%%%%%%%%%%%%%%%%%%%%%%%%%%%%%%%%
%%%%%%%%%%%%%%%%%%%%%%%%%%%%%%%%%%%%%%%%%%%%%%%%%%%%%%%%%%%%%%%%%%%%%%%%%%
In this section, we always write $\pi\in\msn$ in {\it standard
cycle decomposition}, where each cycle is written with its smallest entry first and the
cycles are written in increasing order of their smallest entry.
%and that of fixed points is $\fix(\pi)=\#\{i\in [n]: \pi(i)=i\}$.
%As usual, we denote by $\cyc(\pi)$ the number of cycles of $\pi$.
%Combinatorial statistics on cycle notation have been extensively studied in
%recent years from several points of view (see, e.g.,~\cite{Bre00,Zeng12}).
%We say that $\pi$ is a circular permutation if it has only one cycle.
%Let $A=\{x_1,x_2,\ldots,x_k\}$ be a finite set of positive integers with $k\geq 1$,
%and let $\mc$ be the set of all circular permutations of $A$. We always write
%each circular permutation $w\in\mc$ by using the {\it canonical presentation} $w=(y_1,y_2,\ldots ,y_k)$, where $y_1=\min {A}$, $y_i=w^{i-1}(y_1)$ for $2\leq i\leq k$ and $y_1=w^{k}(y_1)$.
For each $\pi\in\msn$, we say that $\pi$ has an {\it excedance} at $i$ if $\pi(i)>i$.
The excedance number of $\pi$ is defined by
$\exc(\pi)=\#\{i\in [n-1]: \pi(i)>i\}$.
%MacMahon~\cite[vol I, p.~186]{MacMahon1915} observed that
%\begin{equation*}\label{des-exc}
%\#\{\pi\in\msn: \des(\pi)=k\}=\#\{\pi\in\msn: \exc(\pi)=k\}\quad\textrm{for $0\leq k<n$}.
%\end{equation*}
%The classical Eulerian polynomials are defined by
%$$A_n(x)=\sum_{\pi\in\msn}x^{\des(\pi)}=\sum_{\pi\in\msn}x^{\exc(\pi)}.$$
%Any permutation statistic that is equidistributed with $\des$ and $\exc$ is said to be an
%Eulerian statistic. The Eulerian statistics have been extensively studied (see, e.g.,~\cite{Bona12,Branden06,Bre94}).
Following~\cite{Zeng12}, for $\pi\in\msn$, a value $x=\pi(i)$ is called a {\it double excedance} if
$i=\pi^{-1}(x)<x<\pi(x)$, and we say that $x=\pi(i)$ is a {\it cyclic peak} if $i=\pi^{-1}(x)<x>\pi(x)$.
Let $\cpk(\pi)$ denote the number of cyclic peaks of $\pi$.

\begin{definition}
We say that $\pi\in\msn$ is a {\it simsun permutation of the second kind} if for all $k\in [n]$,
after removing the $k$ largest letters of $\pi$, the resulting permutation has no double excedances.
\end{definition}

For example, $(1,5,3,4)(2)$ is not a simsun permutation of the second kind since when we remove the letter 5,
the resulting permutation $(1,3,4)(2)$ contains a double excedance.
Let $\rss_n$ be the set of the simsun permutations of the second kind of length $n$.
It is clear that $\exc(\pi)=\cpk(\pi)$ for $\pi\in\rss_n$.

In the following, we first present a constructive proof of the following identity:
\begin{equation}\label{des-exc}
\sum_{\pi\in\rs_n}x^{\des(\pi)}=\sum_{\pi\in\rss_n}x^{\exc(\pi)}.
\end{equation}

Let $\rss_{n,k}=\{\pi\in\rss_n \mid \exc(\pi)=k\}$.
As a variant of Definition~\ref{def01},
we introduce a definition of {\it labeled simsun permutations of the the second kind}.
\begin{definition}
Let $\sigma\in\rss_{n,k}$.
Suppose $i_1<i_2<\cdots<i_k$ are the excedances of $\sigma$. Then we put the superscript labels $u_r$
right after $i_{r}$, where $1\leq r\leq k$.
In the remaining positions
except the first position of each cycle and the positions right after $\sigma(i_r)$, we put the superscript labels $v_1,v_2,\ldots,v_{n-2k}$
from left to right.
\end{definition}

As an example, for $\sigma=(1,3)(2,4)(5)\in \rss_{5,2}$, the labeled $\sigma$ is given by $(1^{u_1}3)(2^{u_2}4)(5^{v_1})$.

Now we start to construct a bijection, denoted by $\Psi$, between $\rs_{n,k}$ and $\rss_{n,k}$.
When $n=1$, we have $\rs_{1,0}=\{^{y_1}1\}$. Set $\Psi(^{y_1}1)=(1^{v_1})$.
This gives a bijection between $\rs_{1,0}$ and $\rss_{1,0}$.
Let $n=m$. Suppose $\Psi$ is a bijection between $\rs_{m,k}$ and $\rss_{m,k}$ for all $k$.
%Consider the case$n=m+1$.
Given $\pi\in\rs_{m,k}$. Suppose $\Psi(\pi)=\sigma$. Consider the following three cases:
\begin{enumerate}
  \item [\rm ($i$)] If $\widehat{\pi}$ is obtained from $\pi$ by inserting the entry $m+1$ to the position of $\pi$ with label $x_r$, then we insert $m+1$ to $\sigma$ with label $u_r$.
In this case, $\des(\widehat{\pi})=\exc(\Psi(\widehat{\pi}))=k$. Hence $\widehat{\pi}\in\rs_{m+1,k}$ and $\Psi(\widehat{\pi})\in\rss_{m+1,k}$.
 \item [\rm ($ii$)] If $\widehat{\pi}$ is obtained from $\pi$ by inserting the entry $m+1$ to the position of $\pi$ with label $y_r$, then we insert $m+1$ to $\sigma$ with label $v_r$.
In this case, $\des(\widehat{\pi})=\exc(\Psi(\widehat{\pi}))=k+1$. Hence $\widehat{\pi}\in\rs_{m+1,k+1}$ and $\Psi(\widehat{\pi})\in\rss_{m+1,k+1}$.
 \item [\rm ($iii$)] If $\widehat{\pi}$ is obtained from $\pi$ by inserting the entry $m+1$ at the end of $\pi$, then we append $(m+1)$ to $\sigma$ as a new cycle. Hence $\widehat{\pi}\in\rs_{m+1,k}$ and $\Psi(\widehat{\pi})\in\rss_{m+1,k}$.
\end{enumerate}
By induction, we see that $\Psi$ is the desired bijection between $\rs_{m,k}$ and $\rss_{m,k}$ for all $k$,
which also gives a constructive proof of~\eqref{des-exc}.

\begin{ex}
Given $\pi=3412\in\rs_{4,1}$.
The correspondence between $\pi$ and $\Psi(\pi)$ is built up as follows:
\begin{align*}
^{y_1}1&\Leftrightarrow (1^{v_1});\\
^{y_1}1^{y_2}2&\Leftrightarrow (1^{v_1})(2^{v_2});\\
3^{x_1}1^{y_1}2&\Leftrightarrow (1^{u_1}3)(2^{v_1});\\
^{y_1}34^{x_1}1^{y_2}2&\Leftrightarrow (1^{u_1}43^{v_1})(2^{v_2}).
%-2&\rightarrow ^{1}3^{2}3^{-1}1^{3}4^{4}4^{\textbf{-2}}2^{-3}2^{-4}1^{0}\Leftrightarrow 3314455221.
%2&\rightarrow ^{1}3^{\textbf{2}}3^{-1}1^{3}4^{4}4^{5}5^{6}5^{-2}2^{-3}2^{-4}1^{0}\Leftrightarrow 366314455221.
\end{align*}
\end{ex}

%\begin{theorem}
%For any $1\leq k\leq \lrf{n/2}$, we have
%\begin{equation}
%\#\{\pi\in\rs_n: \des(\pi)=k\}=\#\{\pi\in\rss_n: \exc(\pi)=k\}.
%\end{equation}
%Equivalently,
%$$\sum_{\pi\in\rs_n}x^{\des(\pi)}=\sum_{\pi\in\rss_n}x^{\exc(\pi)}.$$
%\end{theorem}
%The number of fixed points of $\pi$ is $\fix(\pi)=\#\{i\in [n]: \pi(i)=i\}$.
%As usual, we denote the number of cycles of $\pi$ by $\cyc(\pi)$.
We now consider the following enumerative polynomials
$$S_n(x,q)=\sum_{\pi\in\rss_n}x^{\exc(\pi)}q^{\cyc(\pi)},$$
where $\cyc(\pi)$ is the number of cycles of $\pi$.
Let $$S=S(x,q;z)=\sum_{n\geq0}S_n(x,q)\frac{z^n}{n!}.$$

\begin{theorem}\label{thm13}
The polynomials $S_n(x,q)$ satisfy the recurrence relation
\begin{equation}\label{Snxq-recu}
S_{n+1}(x,q)=(q+nx)S_n(x,q)+x(1-2x)\frac{\partial}{\partial x}(S_n(x,q)),
\end{equation}
with the initial condition $S_0(x,q)=1$. Furthermore,
\begin{equation}\label{Sxzq}
S(x,q;z)=S(x,z)^q.
\end{equation}
\end{theorem}
\begin{proof}
Let $n$ be a fixed positive integer and given $\sigma\in\rss_n$.
Let $\sigma_i$ be an element of $\rss_{n+1}$ obtained from $\sigma$ by inserting the entry
$n+1$, in the standard cycle decomposition of $\sigma$, right after $i$ if
$i$ is not a cyclic peak of $\sigma$ and $i\in [n]$ or as a new cycle $(n+1)$ if $i=n+1$.
It is clear that
$$ \cyc(\sigma_i)=\left\{
              \begin{array}{ll}
                \cyc(\sigma), & \hbox{if $i\in [n]$;} \\
                \cyc(\sigma)+1, & \hbox{if $i=n+1$.}
              \end{array}
            \right.
$$
Therefore, we have
\begin{align*}
S_{n+1}(x,q)&=\Sigma_{\pi\in\rss_{n+1}}x^{\exc(\pi)}q^{\cyc(\pi)}\\
  &=\Sigma_{i=1}^{n+1}\Sigma_{\sigma\in\rss_{n}}x^{\exc(\sigma_i)}q^{\cyc(\sigma_i)}\\
  &=\Sigma_{\sigma\in\rss_{n}}x^{\exc(\sigma)}q^{\cyc(\sigma)+1}+
\Sigma_{i=1}^{n}\Sigma_{\sigma\in\rss_{n}}x^{\exc(\sigma_i)}q^{\cyc(\sigma)}\\
  &=qS_n(x,q)+\Sigma_{\sigma\in\rss_{n}}(\exc(\sigma)x^{\exc(\sigma)}+(n-2\exc(\sigma))x^{\exc(\sigma)+1})q^{\cyc(\sigma)}\\
  &=qS_n(x,q)+nxS_n(x,q)+\Sigma_{\pi\in\rs_n}(1-2x)\exc(\sigma)x^{\exc(\sigma)}q^{\cyc(\sigma)},\\
\end{align*}
and~\eqref{Snxq-recu} follows.
By rewriting~\eqref{Snxq-recu} in terms of generating function $S$, we have
\begin{equation}\label{Sxz-pde}
(1-xz)S_z=qS+x(1-2x)S_x.
\end{equation}
It is routine to check that the generating function
$\widetilde{S}=\widetilde{S}(x,q;z)=S(x,z)^q$
satisfies~\eqref{Sxz-pde}. Also, this generating function gives $\widetilde{S}(x,q;0)=1,\widetilde{S}(x,0;z)=1$
and $\widetilde{S}(0,q;z)=e^{qz}$. Hence $S=\widetilde{S}$.
\end{proof}

Combining~\eqref{Snxq-recu} and~\cite[Theorem 2]{Ma08}, we get the following corollary.
\begin{cor}
If $q> 0$, then $S_n(x,q)$ has nonpositive and simple zeros for $n\geq 2$.
\end{cor}
%\begin{equation}
%\sum_{n\geq0}S_n(x,y,q)\frac{z^n}{n!}=e^{qz(y-1)}S(x,z)^q.
%\end{equation}
Using~\eqref{Sxzq}, it is easy to verify that
$$S(1,q;z)=\frac{1}{(1-\sin z)^q}.$$
and for $n\geq 1$,
$$S_{n}(x,-1)=\left\{
               \begin{array}{ll}
                 (1-x)(1-2x)^{m-1}, & \hbox{if $n=2m$;} \\
                 -(1-2x)^m, & \hbox{if $n=2m+1$.}
               \end{array}
             \right.$$

A cycle $(b(1),b(2),\ldots)$ is said to be up-down if
it satisfies $b(1)<b(2)>b(3)<\cdots$.
We say that a permutation $\pi$ is {\it cycle-up-down} if it is a product
of up-down cycles.
Let $\triangle_n$ be the set of cycle-up-down permutations in $\msn$.
Deutsch and Elizalde~\cite[p.~193]{Deutsch11} discovered that
$$\sum_{n\geq 0}\sum_{\pi\in\triangle_n}q^{\cyc(\pi)}\frac{z^n}{n!}=S(1,q;z).$$
Therefore, we have
$$\sum_{\pi\in\rss_n}q^{\cyc(\pi)}=\sum_{\pi\in\triangle_n}q^{\cyc(\pi)}.$$

We define $$S_n(x,y,q)=\sum_{\pi\in\rss_n}x^{\exc(\pi)}y^{\fix(\pi)}q^{\cyc(\pi)},$$
where $\fix(\pi)$ is the number of fixed points of $\pi$, i.e., $\fix(\pi)=\#\{i\in [n]: \pi(i)=i\}$.
Note that
\begin{align*}
S_{n}(x,y,q)&=\sum_{i=0}^n\binom{n}{i}(yq-q)^i\sum_{\pi\in\rss_{n-i}}x^{\exc(\pi)}q^{\cyc(\pi)}\\
            &=\sum_{i=0}^n\binom{n}{i}(yq-q)^iS_{n-i}(x,q).
\end{align*}
Using~\eqref{Sxzq}, we obtain
\begin{equation*}\label{Snxyqz-exp}
\sum_{n\geq 0}S_n(x,y,q)\frac{z^n}{n!}=e^{qz(y-1)}S(x,z)^q.
\end{equation*}
%A fixed-point-free permutation is called a {\it derangement}. Let $\DS_n$ denote the set of derangements in $\rss_n$.
%We define
%\begin{align*}
%DS_n(x,q)&=\sum_{\pi\in\DS_n}x^{\exc(\pi)}q^{cyc(\pi)}\\
%DS(x,q;z)&=\sum_{n\geq 0}DS_n(x,q)\frac{z^n}{n!}.
%\end{align*}
%%It follows from~\eqref{Snxyqz-exp} that
%\begin{equation}\label{Dxqz-exp}
%DS(x,q;z)=e^{-qz}S(x,z)^q.
%\end{equation}
%%%%%%%%%%%%%%%%%%%%%%%%%%%%%%%%%%%%%%%%%%%%%%%%%%%%%%%%%%%%%%%%%%%%%%%%%%
\section{Concluding remarks}
%%%%%%%%%%%%%%%%%%%%%%%%%%%%%%%%%%%%%%%%%%%%%%%%%%%%%%%%%%%%%%%%%%%%%%%%%%
%%%%%%%%%%%%%%%%%%%%%%%%%%%%%%%%%%%%%%%%%%%%%%%%%%%%%%%%%%%%%%%%%%%%%%%%%%
%%%%%%%%%%%%%%%%%%%%%%%%%%%%%%%%%%%%%%%%%%%%%%%%%%%%%%%%%%%%%%%%%%%%%%%%%%
In this paper we study the peak statistics on simsum permutations.
It is well known that the descent statistic is equidistributed over $n$-simsun permutations and
$n$-Andr\'e permutations~(see~\cite{Chow11}), and there are bijections between simsun permutations and
increasing 1-2 trees~(see~\cite{Chuang12} for instance).
Therefore, one can find corresponding results on Andr\'e permutations and increasing 1-2 trees.
For example, $S(1,q;z)$ also is the (shifted) exponential generating function that counts Andr\'e permutations with
respect to the size and the number of right-to-left minima (see~\cite[Proposition~1]{Disanto14})
Furthermore, it would be interesting to derive similar results on signed simsum permutations
introduced by Ehrenborg and Readdy~\cite{Ehrenborg98}.


\begin{thebibliography}{22}
\bibitem{Andre84}
D. Andr\'e, \'Etude sur les maxima, minima et s\'equences des permutations, Ann. Sci. \'Ecole Norm. Sup. 3(1) (1884), 121--135.
%%\bibitem{Arnold}
%V.I. Arnold, The calculus of snakes and the combinatorics of Bernoulli, Euler and Springer numbers of Coxeter groups, \newblock {\em Uspekhi Mat. nauk.}, 47(1992), 3--45; Russian Math. Surveys, 47 (1992), 1--51.
%\bibitem{Aguiar04}
%M. Aguiar, N. Bergeron, K. Nyman.  \newblock The peak algebra and the descent algebras of types $B$ and $D$.
%\newblock{\em Trans. Amer. Math. Soc.}, 356(7): 2781--2824, 2004.
\bibitem{Billey13}
S. Billey, K. Burdzy, B.E. Sagan, Permutations with given peak set, \newblock {\em J. Integer Seq.} 16 (2013), Article 13.6.1.

%\bibitem{Bona12}
%M. B\'ona,, Combinatorics of Permutations, Second Edition, Chapman \& Hall/CRC, Boca Raton, Florida, 2012.

%\bibitem{Branden06}
%P. Br\"and\'en, On linear transformations preserving the P\'olya
%frequency property, Trans. Amer. Math. Soc. 358 (2006) 3697--3716.

\bibitem{Branden11}
P. Br\"and\'en and A. Claesson, Mesh patterns and the expansion of permutation statistics as
sums of permutation patterns, \newblock {\em  Electron. J. Combin.} 18(2) (2011), \#P5.

%\bibitem{Bre94}
%F. Brenti, $q$-Eulerian polynomials arising from Coxeter groups,
%European J. Combin. 15 (1994) 417--441.
%\bibitem{Bre00}
%F. Brenti, A class of $q$-symmetric functions arising from plethysm,
%J. Combin. Theory Ser. A 91 (2000) 137--170.
%\bibitem{Borowiec15}
%A. Borowiec, W. M{\L}otkowski, New Eulerian numbers of type $D$, \arxiv{1509.03758}.

\bibitem{Chen11}
W.Y.C. Chen, N.J.Y. Fan, J.Y.T. Jia, Labeled ballot paths and the Springer numbers, \newblock {\em SIAM J. Discrete Math.} 25 (2011), 1530--1546;

%\bibitem{Branden06}
%P. Br\"and\'en, \textit{On linear transformations preserving the P\'olya frequency property,} Trans. Amer. Math. Soc.
%358 (2006), 3697--3716.
%\bibitem{Brenti94}
%F. Brenti, \textit{q-Eulerian polynomials arising from Coxeter groups}, European J. Combin. 15 (1994), 417--441.
\bibitem{Chow11}
C-O. Chow, W. C. Shiu, Counting simsun permutations by descents, \newblock {\em Ann. Comb.} 15 (2011), 625--635.

\bibitem{Chuang12}
W.-C. Chuang, S.-P. Eu, T.-S. Fu, Y.-J. Pan, On simsun and double simsun permutations avoiding a pattern of length three,
\newblock {\em Fund. Inform.} 117 (2012), 155--177.


\bibitem{Deutsch11}
E. Deutsch and S. Elizalde, Cycle-up-down permutations, \newblock {\em Australas. J. Combin.} 50 (2011), 187--199.

\bibitem{Deutsch12}
E. Deutsch, S. Elizalde, Restricted simsun permutations,
\newblock {\em Ann. Combin.} 16(2) (2012), 253--269.

\bibitem{Disanto14}
F. Disanto, Andr\'e permutations, right-to-left and left-to-right minima,
\newblock {\em S\'eminaire Lothar. Combin.} 70 (2014), Article B70f.


%\bibitem{Dilks09}
%K. Dilks, T.K. Petersen, J.R. Stembridge. \newblock Affine descents and the Steinberg torus. \newblock {\em Adv. in
%Appl. Math.}, 42: 423--444, 2009.
\bibitem{Ehrenborg98}
R. Ehrenborg, M. Readdy, Coproducts and the cd-index, \newblock {\em J. Algebraic Combin.} 8 (1998), 273--299.

%\bibitem{Eisenstein03}
%M. A. Eisenstein-Taylor, Polytopes, permutation shapes and bin packing, \newblock {\em Adv. in Appl. Math.}, 30 (2003), 96--109.

%\bibitem{Elizalde03}
%S. Elizalde and M. Noy, Consecutive patterns in permutations, \newblock {\em Adv. in Appl. Math.} 30 (2003), 110--125.

\bibitem{Eu14}
S.-P. Eu, T.-S. Fu, Y.-J. Pan, A refined sign-balance of simsun permutations,
\newblock {\em European J. Combin.} 36 (2014), 97--109.

\bibitem{Foata73}
D. Foata and M.P. Sch\"utzenberger, Nombres d'Euler et permutations alternantes, in A Survey of Combinatorial
Theory, J.N. Srivastava et al. (Eds.), Amsterdam, North-Holland, 1973, pp. 173--187.

\bibitem{Foata01}
D. Foata, G.-N. Han, Arbres minimax et polyn\^{o}mes d'Andr\'e, \newblock {\em Adv. in Appl. Math.}, 27 (2001),
367--389.

%\bibitem{Foata16}
%D. Foata and G.-N. Han, Andr\'e permutation calculus; a twin Seidel matrix sequence, \arxiv{1601.04371}.
%\bibitem{g60}
%F.R. Gantmacher, The Theory of Matrices, vol. II, Chelsea, New York, 1960.
%\bibitem{Firro07}
%G. Firro, T. Mansour, M.C. Wilson, Longest alternating subsequences in pattern-restricted permutations,
%\newblock {\em Electron. J. Combin.}, 14 (2007), \#R34.

\bibitem{gw96}
J.~Garloff, D.G.~Wagner, \newblock Hadamard products of stable polynomials are stable, \newblock {\em  J. Math. Anal. Appl.} 202 (1996), 797--809.
%\bibitem{Hetyei96}
%G. Hetyei, On the $\cd$-variation polynomials of Andr\'e and simsun permutations, \newblock {\em Discrete Comput. Geom.}, 16 (1996), 259--275

\bibitem{Hetyei98}
G. Hetyei, E. Reiner, Permutation trees and variation statistics, \newblock {\em European J. Combin.} 19 (1998), 847--866.

\bibitem{Ma121}
S.-M. Ma. \newblock Derivative polynomials and enumeration of permutations by number of interior and left peaks.
\newblock {\em Discrete Math.} 312 (2012), 405--412.

\bibitem{Ma122}
S.-M. Ma, \newblock An explicit formula for the number of permutations with a given number of alternating runs, \newblock {\em J. Combin. Theory Ser. A} 119 (2012), 1660--1664.

\bibitem{Ma123}
S.-M. Ma, \newblock Enumeration of permutations by number of alternating runs, \newblock {\em Discrete Math.} 313 (2013), 1816--1822.

\bibitem{Ma08}
S.-M. Ma, Y. Wang, $q$-Eulerian polynomials and polynomials with only real zeros, \newblock {\em Electron. J. Combin.} 15 (2008), R17.
%\bibitem{Purtill93}
%M. Purtill, Andr\'e permutations, lexicographic shellability and the cd-index of a convex polytope, Trans. Amer. Math. Soc. 338 (1993) 77--104 .
%\bibitem{Pak15}
%I. Pak, R. Pemantle, On the longest $k$-alternating subsequence, \newblock {\em Electron. J. Combin.} 22(1)
%(2015), \#P1.48.

%\bibitem{MacMahon1915}
%P.A. MacMahon, Combinatory Analysis, 2 volumes, Cambridge University Press, London,
%1915--1916. Reprinted by Chelsea, New York, 1960.

\bibitem{Zeng12}
H. Shin, J. Zeng, \newblock The symmetric and unimodal expansion of Eulerian polynomials via continued fractions, \newblock {\em European J. Combin.} 33 (2012), 111--127.


\bibitem{Sloane}
N.J.A. Sloane, The On-Line Encyclopedia of Integer Sequences,
published electronically at
http://oeis.org, 2010.

\bibitem{Springer}
T.A. Springer, Remarks on a combinatorial problem, \newblock {\em Nieuw Arch. Wisk.} 19 (1971), 30--36.

\bibitem{Sta08}
R.P. Stanley, Longest alternating subsequences of permutations,
\newblock {\em Michigan Math. J.} 57 (2008), 675--687.

\bibitem{Sundaram1994}
S. Sundaram, The homology representations of the symmetric group on Cohen-Macaulay subposets of the partition lattice, \newblock {\em Adv. Math.} 104 (1994), 225--296.

\bibitem{Sundaram1995}
S. Sundaram, The homology of partitions with an even number of blocks, \newblock {\em J. Algebraic Combin.} 4 (1995),
69--92.

\bibitem{Sundaram1996}
S. Sundaram, Plethysm, partitions with an even number of blocks and Euler numbers, DIMACS Series
in Discrete Mathematics and Theoretical Computer Science 24, (Billera, Greene, Simion,
Stanley, eds.), Amer. Math. Soc. (1996), 171--198.
%\bibitem{Bona12}
%M. B\'ona,, Combinatorics of Permutations, Second Edition, Chapman \& Hall/CRC, Boca Raton, Florida, 2012.
%\bibitem{Janson11}
%S. Janson, M. Kuba and A. Panholzer, Generalized Stirling permutations, families of increasing trees and urn models,
%\newblock {\em J. Combin. Theory Ser. A,} 118 (2011) 94--114.
%\bibitem{Sloane}
%N.J.A. Sloane, The On-Line Encyclopedia of Integer Sequences,
%published electronically at
%http://oeis.org, 2010.
\end{thebibliography}
\end{document}